\documentclass[12pt,a4paper]{amsart}
\usepackage{mathrsfs}

\newtheorem{theo+}              {Theorem}           [section]
\newtheorem{prop+}  [theo+]     {Proposition}
\newtheorem{coro+}  [theo+]     {Corollary}
\newtheorem{lemm+}  [theo+]     {Lemma}
\newtheorem{exam+}  [theo+]     {Example}
\newtheorem{rema+}  [theo+]     {Remark}
\newtheorem{defi+}  [theo+]     {Definition}

\newenvironment{theorem}{\begin{theo+}}{\end{theo+}}
\newenvironment{proposition}{\begin{prop+}}{\end{prop+}}
\newenvironment{corollary}{\begin{coro+}}{\end{coro+}}

\usepackage{amsthm}
\theoremstyle{plain} \theoremstyle{remark}
\newtheorem{remark}{Remark}

\newtheorem*{ack}{\bf Acknowledgments}

\def \r{\mbox{${\mathbb R}$}}
\def\E{/\kern-1.0em \equiv }
\def\h{harmonic morphism}

\title{Constant mean curvature and totally umbilical
biharmonic surfaces in 3-dimensional geometries}
\author{Ye-Lin Ou$^{*}$ and Ze-Ping Wang$^{**}$}

\address{Department of
Mathematics,\newline\indent Texas A $\&$ M University-Commerce,
\newline\indent Commerce TX 75429,\newline\indent USA.\newline\indent
E-mail:yelin$\_$ou@tamu-commerce.edu \;(Ou)\newline \indent \vskip
0.1cm Department of Mathematics $\&$ Physics,\newline\indent Yunnan
Wenshan University,\newline\indent No. 66 Xuefu Road Wenshan County
Wenshan, Yunnan 663000,\newline\indent People's Republic of China
\newline\indent E-mail:zeping.wang@gmail.com \;(Wang) }

\thanks{* Supported by Texas A $\&$ M University-Commerce ``Faculty
Research Enhancement Project" (2010-11) .\\
\indent** Supported by Yunnan Wenshan University Research Project
09WSY03. The author is also grateful to the Department of
Mathematics, Texas A $\&$ M University-Commerce for the hospitality
he received during a visit in Fall 2009 during which a part of this
work was done.}
\begin{document}
\title[Biharmonic surfaces in 3-dimensional geometries] {Constant mean curvature and totally umbilical
biharmonic surfaces in 3-dimensional geometries}
\date {01/17/2011}  \maketitle

\section*{Abstract}
\begin{quote}
{\footnotesize We prove that a totally umbilical biharmonic surface
in any $3$-dimensional Riemannian manifold has constant mean
curvature. We use this to show that a totally umbilical surface in
Thurston's 3-dimensional geometries is proper biharmonic if and only
if it is a part of $S^2(1/\sqrt{2})$ in $S^3$. We also give complete
classifications of constant mean curvature proper biharmonic
surfaces in 3-dimensional geometries and in 3-dimensional
Bianchi-Cartan-Vranceanu spaces, and a complete classifications of
proper biharmonic Hopf cylinders in 3-dimensional
Bianchi-Cartan-Vranceanu spaces.}
\end{quote}
\section{Introduction and preliminaries}

We assume that all manifolds, maps, tensor fields and other objects studied in this paper are smooth.\\

 A map $\varphi:(M, g)\longrightarrow (N, h)$ between
Riemannian manifolds is  {\bf biharmonic map} if $\varphi|\Omega$ is
a critical point of the bienergy
\begin{equation}\nonumber
E^{2}\left(\varphi,\Omega \right)= \frac{1}{2} {\int}_{\Omega}
\left|\tau(\varphi) \right|^{2}{\rm d}x
\end{equation}
for every compact subset $\Omega$ of $M$, where $\tau(\varphi)={\rm
Trace}_{g}\nabla {\rm d} \varphi$ is the tension field of $\varphi$.
Locally, biharmonic maps are solutions of the following system of
4th order PDEs:
\begin{equation}\notag
{\rm
Trace}_{g}(\nabla^{\varphi}\nabla^{\varphi}-\nabla^{\varphi}_{\nabla^{M}})\tau(\varphi)
- {\rm Trace}_{g} R^{N}({\rm d}\varphi, \tau(\varphi)){\rm d}\varphi
=0,
\end{equation}
where $R^{N}$ denotes the curvature operator of $(N, h)$ defined by
$$R^{N}(X,Y)Z=
[\nabla^{N}_{X},\nabla^{N}_{Y}]Z-\nabla^{N}_{[X,Y]}Z.$$

As any harmonic map (the one with $\tau(\varphi)=0$) is biharmonic
we use the name {\bf proper biharmonic} for those
biharmonic maps which are not harmonic.\\

A submanifold is a {\bf biharmonic submanifold} if the isometric
immersion that defines the submanifold is a biharmonic map.
Biharmonic submanifolds include  minimal submanifolds as a subset as
it is well known that an isometric immersion is harmonic if and only
if it is minimal. We use {\bf proper biharmonic submanifolds} to
name
those biharmonic submanifolds which are not minimal.\\

Many recent works in the geometric study of biharmonic maps have
been focused on the following two fundamental problems: {\bf (1)
existence problem:} given two model spaces (e.g., some ``good"
spaces such as spaces of constant sectional curvature or more
general symmetric or homogeneous spaces), does there exist a proper
biharmonic map mapping one space into another? {\bf (2)
classification problem:} classify all proper biharmonic maps between
two model spaces where the existence is known. Some
typical and challenging classification problems are the following \\

{\bf Chen's conjecture \cite{CH}:} any biharmonic submanifold in a
Euclidean
space is minimal, and\\

{\bf The generalized Chen's conjecture}: any biharmonic submanifold
of $(N, h)$ with ${\rm Riem}^N\leq 0$ is minimal (see e.g.,
\cite{CMO1}, \cite{MO}, \cite{BMO1}, \cite{BMO2}, \cite{BMO3},
\cite{Ba1}, \cite{Ba2}, \cite{Ou1}, \cite{Ou2}, \cite{IIU}).\\

For some recent progress of classifications of biharmonic
submanifolds we refer the readers to \cite{CMO1}, \cite{CMO2},
\cite{BMO1}, \cite{BMO2}, \cite{BMO3}, \cite{MO}, and the
references therein.\\

In the recent paper \cite{Ou1}, the first named author of this paper
derived the equation for biharmonic hypersurfaces in a generic
Riemannian manifold which can be stated in the following
\begin{theorem}\cite{Ou1}\label{MTH}
Let $\varphi:M^{m}\longrightarrow N^{m+1}$ be an isometric immersion
of codimension-one with mean curvature vector $\eta=H\xi$. Then
$\varphi$ is biharmonic if and only if:
\begin{equation}\label{BHEq}
\begin{cases}
\Delta H-H |A|^{2}+H{\rm
Ric}^N(\xi,\xi)=0,\\
 2A\,({\rm grad}\,H) +\frac{m}{2} {\rm grad}\, H^2
-2\, H \,({\rm Ric}^N\,(\xi))^{\top}=0,
\end{cases}
\end{equation}
where ${\rm Ric}^N : T_qN\longrightarrow T_qN$ denotes the Ricci
operator of the ambient space defined by $\langle {\rm Ric}^N\, (Z),
W\rangle={\rm Ric}^N (Z, W)$ and  $A$ is the shape operator of the
hypersurface with respect to the unit normal vector $\xi$.
\end{theorem}

An nice application of the above equation was made in \cite{OT}
where the equation was use to determine a conformally flat metric on
$\r^{5}$ so that a foliation by the hyperplanes defined by the
graphs of linear functions becomes a proper biharmonic foliation.
Those proper biharmonic hyperplanes were eventually used to
construct counter examples to prove that the generalized Chen's
conjecture is false (see \cite{OT} for details).\\

In this paper, we will use equation (\ref{BHEq}) to study biharmonic
surfaces in Thurston's 3-dimensional geometries. We first show that
a totally umbilical biharmonic surface in any $3$-dimensional
Riemannian manifold has constant mean curvature. We then use this to
show that the only totally umbilical proper biharmonic surface in
3-dimensional geometries is a part of $S^2(1/\sqrt{2})$ in $S^3$. We
also show that the only constant mean curvature proper biharmonic
surface in 3-dimensional geometries are a part of $S^2(1/\sqrt{2})$
in $S^3$ or a part of $S^1(1/\sqrt{2})\times \mathbb{R}$ in
$S^2\times \mathbb{R}$, and the only constant mean curvature proper
biharmonic surfaces in a 3-dimensional Bianchi-Cartan-Vranceanu
space is apart of $S^2(1/\sqrt{2})$ in $S^3$ or a part of a Hopf
cylinder in $S^2(1/(2\sqrt{m})\times \mathbb{R}$ or $SU(2)$ whose
base curve is a circle with radius $R=1/\sqrt{8m-l^2\:}$ in the base
sphere $S^2(\frac{1}{2\sqrt{m}})$ identified with
$\left(\mathbb{R}^{2},h=\frac{dx^2+dy^2}{[1+m(x^2+y^2)]^2}\right)$.

\section{Constant mean curvature biharmonic surfaces in 3-dimensional
geometries}

It is well known that Thurston's eight models for 3-dimensional
geometries consist of : 3-dimensional space forms $ \r^3, \; S^3,
H^3$, the product spaces: $S^2\times \r,\; H^2\times \r$ and
$\widetilde{SL}(2,\r)$, ${\rm Nil}$, \; $\rm Sol$.

It is also known (see, e.g., \cite{BDI}, \cite{CMOP}) that
Bianchi-Cartan-Vranceanu 3-dimensional spaces :
\begin{equation}\label{CV}
M^3_{m,l}=\left(\mathbb{R}^{3},g=\frac{dx^2+dy^2}{[1+m(x^2+y^2)]^2}+[dz+\frac{l}{2}\frac{y
dx-x dy}{1+m(x^2+y^2)}]^2\right)
\end{equation}

include six of Thurston's eight 3-dimensional geometries in the
family except for the
hyperbolic space $H^3$ and $\rm Sol$.\\

As biharmonic surfaces in 3-dimensional space forms have been
completely classified (\cite{Ji}, \cite{CI}, \cite{CMO1}), we can
obtain the classification of constant mean curvature biharmonic
surfaces in 3-dimensional geometries by classifying constant mean
curvature biharmonic surfaces in Bianchi-Cartan-Vranceanu
3-dimensional spaces and in $\rm
Sol$.\\

{\bf 2.1 Constant mean curvature biharmonic surfaces in
Bianchi-Cartan-Vranceanu 3-spaces}\\

We adopt the following notations and sign convention for Riemannian
curvature operator:
\begin{equation}\notag
 R(X,Y)Z=\nabla_{X}\nabla_{Y}Z
-\nabla_{Y}\nabla_{X}Z-\nabla_{[X,Y]}Z,
\end{equation}
and the Riemannian and the Ricci curvatures:
\begin{equation}\notag
\begin{array}{lll}
&&  R(X,Y,Z,W)=g( R(Z,W)Y,X),\\\notag && {\rm Ric}(X,Y)= {\rm
Trace}_{g}R=\sum_{i=1}^m R(Y, e_i, X, e_i)=\sum_{i=1}^m \langle R(
X,e_i) e_i, Y\rangle.
\end{array}
\end{equation}

For a Bianchi-Cartan-Vranceanu 3-space given in (\ref{CV}), one can
easily check  that the vector fields
\begin{equation}\notag
E_{1}=F\frac{\partial}{\partial
x}-\frac{ly}{2}\frac{\partial}{\partial z},\;E_{2}=F
\frac{\partial}{\partial y}+\frac{lx}{2}\frac{\partial}{\partial
z},\;E_{3}=\frac{\partial}{\partial z},
\end{equation}
where $F=1+m(x^2+y^2)$, form an orthonormal
frame.\\
A straightforward computation shows that (see also \cite{CMOP})
\begin{equation}\label{Lie}
[E_1,E_2]=2mxE_{2}-2myE_{1}+lE_{3},\;\; {\rm all\;\;
other}\;\;[E_i,E_j]=0,\;\;i,j=1, 2, 3.
\end{equation}

Let $\nabla$ denote the Levi-Civita connection of a 3-dimensional
Bianchi-Cartan-Vranceanu space, then we can check (see also
\cite{BDI} and \cite{CMOP}) that
\begin{equation}\label{CNil}
\begin{cases}
\nabla_{E_{1}}E_{1}=2myE_{2},\;\;\nabla_{E_{2}}E_{2}=2mxE_{1},\;\;\\
\nabla_{E_{1}}E_{2}=-2myE_{1}+\frac{l}{2}E_{3},\;\;
\nabla_{E_{2}}E_{1}=-2mxE_{2}-\frac{l}{2}E_{3},\;\;\\
\nabla_{E_{3}}E_{1}=\nabla_{E_{1}}E_{3}=-\frac{l}{2}E_{2},\;\;\nabla_{E_{3}}E_{2}=\nabla_{E_{2}}E_{3}=\frac{l}{2}E_{1},\;\;\\
\;\; all \;\; other\;\;
\nabla_{E_i}E_j=0,\;i,j=1, 2, 3.\;\;\\
\end{cases}
\end{equation}

A further computation (see also \cite{BDI} and \cite{CMOP}) gives
the possible nonzero components of the curvatures:
\begin{equation}\label{BCV1}
\begin{array}{lll}
 R_{1212}=g(R(E_{1},E_{2})E_{2},E_{1})=4m-\frac{3l^2}{4},\\
R_{1313}=g(R(E_{1},E_{3})E_{3},E_{1})=\frac{l^2}{4},\\
R_{2323}=g(R(E_{2},E_{3})E_{3},E_{2})=\frac{l^2}{4},
\end{array}
\end{equation}
and the Ricci curvature:
\begin{eqnarray}\label{RicNil}
&& {\rm Ric}\, (E_{1},E_{1})={\rm Ric}\,
(E_{2},E_{2})=4m-\frac{l^2}{2},\;\;\\\notag &&{\rm Ric}\,
(E_{3},E_{3})=\frac{l^2}{2},\\\notag && {\rm all\;other }\; {\rm
Ric}\, (E_i,E_j)=0,\;i\neq j.
\end{eqnarray}

\begin{theorem}\label{CVV}
A constant mean curvature surface in a $3$-dimensional
 Bianchi-Cartan-Vranceanu space is proper biharmonic if and only if it is a
part of: \\
$(1)$\;\; $S^2(1/\sqrt{2m})$ in $S^3(1/\sqrt{m})$, or \\
$(2)$\;\; $S^1((1/(2\sqrt{2m}))\times \mathbb{R}$ in
$S^2(1/(2\sqrt{m})\times \mathbb{R}$, or\\
$(3)$\;\; a Hopf cylinder in $SU(2)$ with $4m-l^2>0$ over a circle
of radius $R=1/\sqrt{8m-l^2\:}$ in the base sphere
$S^2(\frac{1}{2\sqrt{m}})$ identified with
$\left(\mathbb{R}^{2},h=\frac{dx^2+dy^2}{[1+m(x^2+y^2)]^2}\right)$.
\end{theorem}
\begin{proof}
If the mean curvature $H$ is constant, then the biharmonic equation
reduces to
\begin{equation}\label{hn45}
\begin{cases}
-H |A|^{2}+H{\rm
Ric}^N(\xi,\xi)=0,\\
 \, H \,({\rm
Ric}\,(\xi))^{\top}=0.
\end{cases}
\end{equation}
Let $\{e_{i}=a_{i}^{\gamma} E_{\gamma},\;\;
\xi=c^{\gamma}E_{\gamma},\;\;i=1, 2\}$ be an  orthonormal frame on
the ambient space adapted to the surface with $\xi$ being the unit
normal vector field of the surface. Then, a straightforward
computation using (\ref{RicNil}) gives
\begin{eqnarray}\label{hn46}
({\rm Ric}\,(\xi))^{\top}&=&\sum_{i=1}^{2}{\rm
Ric}(c^{\gamma}E_{\gamma},a^{\gamma}_iE_{\gamma})e_i\\\notag &
=&\sum_{i=1}^{2}
\left[(4m-\frac{l^2}{2})\sum_{\gamma=1}^{2}c^{\gamma}a_i^{\gamma}+\frac{l^2}{2}c^{3}a_i^{3}\right]e_i\\\notag
&=&(l^2-4m)c^{3}\sum_{i=1}^{2}a_i^{3}e_i,\\\label{hn47}
 {\rm Ric}^N(\xi,\xi)&=&{\rm Ric}^N(c^{\gamma}E_{\gamma},c^{\gamma}E_{\gamma})=\sum_{i=1}^{2}(4m-\frac{l^2}{2})(c^{i})^2+\frac{l^2}{2}(c^{3})^2
 \\\notag& =&(4m-\frac{l^2}{2})+(l^2-4m)(c^{3})^2.
\end{eqnarray}
Substituting (\ref{hn46}) and (\ref{hn47}) into (\ref{hn45}) we
conclude that the constant mean curvature surface is biharmonic if
and only if
\begin{equation}\notag
\begin{cases}
-H [|A|^{2}-(4m-\frac{l^2}{2})-(l^2-4m)(c^{3})^2]=0,\\
(l^2-4m)c^{3}a^{3}_{1}H=0,\\
(l^2-4m)c^{3}a^{3}_{2}H=0,
\end{cases}
\end{equation}
which has solution $H=0$ meaning that the surface is minimal, or
\begin{equation}\label{Chn}
\begin{cases}
|A|^{2}-(4m-\frac{l^2}{2})-(l^2-4m)(c^{3})^2=0,\\
(l^2-4m)c^{3}a^{3}_{1}=0,\\
(l^2-4m)c^{3}a^{3}_{2}=0.
\end{cases}
\end{equation}
We can solve Equation (\ref{Chn}) by considering the following
cases:\\

{\bf Case I:} $l^2-4m=0$. In this case, $|A|^2=\frac{l^2}{2}$, and
the corresponding Bianchi-Cartan-Vranceanu 3-space is locally either
$\r^3$ or $\mathbb{S}^3(1/\sqrt{m})$ and, by the classification of
biharmonic surfaces in $3$-dimensional space form (see \cite{Ji},
\cite{CI}, \cite{CMO1}), we conclude that in these cases, the only
proper
biharmonic surface is a part of $S^2(1/\sqrt{2m})$ in $S^3(1/\sqrt{m})$.\\

{\bf Case II:} $l^2-4m\ne 0$. In this case, by the last two
equations
of (\ref{Chn}),  we have either $c^3=0$ or  $c^3\ne 0$.\\

For Case II-A: $c^3=0$, we use the first equation of (\ref{Chn}) to
conclude that
\begin{equation}\label{TB}
|A|^2=4m-\frac{l^2}{2}.
\end{equation}

Noting that $c^3=0$ means that the normal vector field of the
surface $\Sigma$ is always orthogonal to
$E_3=\frac{\partial}{\partial z}$ so we can take an another
orthonormal frame $\{e_{1}=aE_{1}+bE_2,\;\;e_2=E_3\;\;
\xi=bE_{1}-aE_2,\}$ adapted to the surface with $a^2+b^2=1$ and
$\xi$ being the unit normal vector filed. Using (\ref{CNil}) we can
compute (see also \cite{Ve} Example 3.4.1)
\begin{equation}\label{cbu1}
\begin{array}{lll}
\nabla_{e_{1}}\xi&=&\{a e_{1}(b)-b e_{1}(a)+2
m(ay-bx)\}e_{1}-\frac{l}{2}e_{2},\\
\nabla_{e_{2}}\xi&=&\{a e_{2}(b)-b e_{2}(a)-\frac{l}{2}\}e_{1}.
\end{array}
\end{equation}
A further computation gives the second fundamental form of the
surface with respect to the chosen adapted orthonormal frame:
 \begin{equation}\label{cbu2}
\begin{array}{lll}
h(e_1,e_1)=-\langle\nabla_{e_{1}}\xi,e_{1}\rangle=-(a
e_{1}(b)-b e_{1}(a)+2 m(ay-bx)),\\
h(e_1,e_2)=-\langle\nabla_{e_{1}}\xi,e_{2}\rangle=\frac{l}{2},\\h(e_2,e_1)
=-\langle\nabla_{e_{2}}\xi,e_{1}\rangle=\frac{l}{2}-a
e_{2}(b)+b e_{2}(a),\\
h(e_2,e_2)=0,
\end{array}
\end{equation}
It follows from (\ref{cbu2}), the symmetry $h(e_1,e_2)=h(e_2,e_1)$,
and $0=e_{2}(a^2+b^2)=2ae_{2}(a)+2be_{2}(b)$ that
\begin{equation}\label{cbu3}
\begin{array}{lll}
e_{2}(a)=e_{2}(b)=0,
\end{array}
\end{equation}
which means the functions $a, b$ are constant along the fibers of
the Riemannian submersion
\begin{eqnarray}\label{RSM}
\pi:
M^3_{m,l}\longrightarrow\left(\mathbb{R}^{2},h=\frac{dx^2+dy^2}{[1+m(x^2+y^2)]^2}\right),\;\;\;
\pi(x,y,z)=(x, y).
\end{eqnarray}

On the other hand, it is not difficult to see that the integral
curves of $e_2$ are geodesics on the surface $\Sigma$. It follows
from a well-known fact in the differential geometry of surfaces that
we can parametrize $\Sigma$  by $r=r(u,v)$ so that the $u-curves$
are the integral curves of $e_2$ and the $v-curves$ are the
orthogonal trajectories of the $u-curves$. Let
$\gamma:I\longrightarrow M^3_{m,l}$,\; $\gamma=\gamma (s)$ be a
$v-curve$ on the surface with arclength parameter, then it is
horizontal with respect to the Riemannian submersion $\pi$. Let
$\alpha(s)=\pi(\gamma(s))$ be the curve in the base space of the
Riemannian submersion, then the surface $\Sigma$ can be viewed as
$\cup_{s\in I}\pi^{-1}(\alpha(s))$, a Hopf cylinder over the curve
$\alpha(s) \;\subset\;
\left(\mathbb{R}^{2},h=\frac{dx^2+dy^2}{[1+m(x^2+y^2)]^2}\right)$.\\

If we write $\alpha(s)=(x(s), y(s))$, then the surface $\Sigma$ can
be parametrized as $r(s,t)=(x(s),y(s), t)$ since the fiber  of $\pi$
over a point $(x_0,y_0)$ is $\pi^{-1}(x_0,y_0)=\{(x_0,y_0, t)|t\in
\r\}$.

Let $\kappa_g$ denote the geodesic curvature of the base curve.
Then, we can use the Frenet formula to check (see also \cite{Ve},
Example 3.4.1) that $\kappa_g=-(a e_{1}(b)-b e_{1}(a)+2 m(ay-bx))$.
It follows from Equations (\ref{cbu2}) and (\ref{cbu3}) that
\begin{equation}\label{q1}
\begin{array}{lll}
|A|^2=\kappa_g^2+l^2/2,\\
H=\kappa_g/2.
\end{array}
\end{equation}
Combining (\ref{q1}) and (\ref{TB}) we have
\begin{equation}\label{q2}
\begin{array}{lll}
\kappa_g^2=4m-l^2>0,\\
H^2=(4m-l^2)/4>0.
\end{array}
\end{equation}

To have a geometric characterization of the base curve we first
notice that the proper biharmonicity ($H^2> 0$) of the Hopf cylinder
implies that $m>0$. It follows that the potential BCV space has to
be either $S^2(1/(2\sqrt{m})\times \mathbb{R}$ or $SU(2)$ with
$m>0$. Using the well-known curvature relation (the O'Neill's
formula) of the Riemannian submersion (\ref{RSM}) we conclude that
the base space must have positive curvature $4m$ so it can be viewed
as a sphere $S^2(\frac{1}{2\sqrt{m}})$. As the curve in the base
sphere $S^2(\frac{1}{2\sqrt{m}})$ has constant geodesic curvature
$\kappa_g=\sqrt{4m-l^2\;}$ one can check that this curve, viewed as
a curve in Euclidean 3-space of which $S^2(\frac{1}{2\sqrt{m}})$ is
a subset, has curvature
$\kappa=\sqrt{\kappa_n^2+\kappa_g^2\;}=\sqrt{8m-l^2}$ and torsion
$\tau=-\frac{2\sqrt{m}\;\kappa'}{\kappa\kappa_g}=0$. From this we
conclude that the base curve of the Hopf cylinder is a circle on
$S^2(\frac{1}{2\sqrt{m}})$ with radius $\frac{1}{\sqrt{8m-l^2}}$. In
particular, when $l=0$ we obtain the Hopf cylinder
$S^1((1/(2\sqrt{2m}))\times \mathbb{R}$ in $S^2(1/(2\sqrt{m})\times
\mathbb{R}$, and when $l=0, m=1/4$ the proper Hopf cylinder
$S^1(1/\sqrt{2})\times \mathbb{R}$ in $S^2\times
\mathbb{R}$ found in \cite{Ou1}.\\

For Case II-B: $c^3\ne 0$, we use the last two equations of
(\ref{Chn}) to conclude that $a^3_1= a_2^3=0$. It follows that ${\rm
Span}\{e_1,\,e_2\}={\rm Span}\{E_1,\,E_2\}$. This means the
distribution determined by $\{E_1, E_2\}$ is integrable and hence
(by Frobeniu's theorem) is involutive.  It follows from (\ref{Lie})
that $l=0$. It also follows that $\xi=\pm E_3$ and hence $c^3=\pm
1$. Substituting $l=0$ and $c^3=\pm 1$ into the first equation of
(\ref{Chn}) we obtain $|A|^2=0$ which means that the surface is
totally geodesic.\\

Combining the results proved above we obtain the theorem.
\end{proof}
\begin{remark}
 It is interesting to know that it is shown in \cite{In} that
there exists proper biharmonic Hopf cylinder in a Sasakian
3-manifold of constant holomorphic sectional curvature $c=4m-3>1$
which, according to \cite{Ta}, is isometric to a
Bianchi-Cartan-Vranceanu space with $l=2$, i.e., a $SU(2)$. On the
other hand, using a different method, the authors in \cite{FO} gave
an explicit equation of a Hopf cylinder in Bianchi-Cartan-Vranceanu
space with $l=2$ and $c=4m-3>1$. Our results show that this is the
only proper biharmonic surface in $SU(2)$ with constant mean
curvature. Clearly, the proper biharmonic Hopf cylinder is not
totally umbilical.
\end{remark}
Now we are ready to prove the following theorem which gives a
complete classification a proper biharmonic cylinder in
Bianchi-Cartan-Vranceanu spaces.
\begin{theorem}\label{BCV}
Let $\pi:
M^3_{m,l}=\left(\mathbb{R}^{3},g=\frac{dx^2+dy^2}{[1+m(x^2+y^2)]^2}+[dz+\frac{l}{2}\frac{y
dx-x dy}{1+m(x^2+y^2)}]^2\right)\longrightarrow
(\mathbb{R}^{2},h=\frac{dx^2+dy^2}{[1+m(x^2+y^2)]^2})$,
$\pi(x,y,z)=(x, y)$ be a Riemannian submersion. Let
$\alpha:I\to(\r^2, h=\frac{dx^2+dy^2}{[1+m(x^2+y^2)]^2})$ be
 an immersed regular curve parametrized by arclength. Then the Hopf cylinder
 $\Sigma=\cup_{s\in I}\pi^{-1}(\alpha(s))$ is a proper biharmonic surface in
Bianchi-Cartan-Vranceanu space if and only if it is a
part of: \\
$(1)$ \;\; $S^1((1/(2\sqrt{2m}))\times \mathbb{R}$ in $S^2(1/(2\sqrt{m})\times \mathbb{R}$, or\\
$(2)$\;\; a Hopf cylinder in $SU(2)$ with $4m-l^2>0$ over a circle
of radius $R=1/\sqrt{8m-l^2\:}$ in the base sphere
$S^2(\frac{1}{2\sqrt{m}})$ identified with
$\left(\mathbb{R}^{2},h=\frac{dx^2+dy^2}{[1+m(x^2+y^2)]^2}\right)$.
\end{theorem}
\begin{proof}
 Let $\alpha:I\to(\r^2, h=\frac{dx^2+dy^2}{[1+m(x^2+y^2)]^2})$
 =$S^2(\frac{\sqrt{m}}{2m})$ with $\alpha(s)=(x(s),y(s))$ be
 an immersed regular curve parametrized by arclength with
the geodesic curvature $\kappa_{g}$.  As in \cite{Ou1} we can take
the horizontal lifts of the tangent and the principal normal vectors
of the curve $\alpha$:  $X=\frac{x'}{F}E_1+\frac{y'}{F}E_2$ and
$\xi=\frac{y'}{F}E_1-\frac{x'}{F}E_2$ (where $F=1+m(x^2+y^2)$)
together with $V=E_3$ to be an orthonormal frame adapted to the Hopf
cylinder. A straightforward computation gives:
\begin{equation}\label{Rict11}
\begin{cases}
 {\rm Ric}\, (\xi,\xi)=(4m-\frac{l^2}{2})(\frac{x'^2+y'^2}{F^2})=4m-\frac{l^2}{2},\\
{\rm Ric}\,
(\xi,X)=(4m-\frac{l^2}{2})(\frac{-x'y'+x'y'}{F^2})=0,\\{\rm Ric}\,
(\xi,V)={\rm Ric}\, (\frac{y'}{F}E_1-\frac{x'}{F}E_2,E_3)=0,
\end{cases}
\end{equation}
and the torsion of the lifting curve $\pi^{-1}(\alpha(s))$
\begin{eqnarray}\label{Rict21}
\tau_{g}=-\langle\nabla_{X}V,\xi\rangle=-\langle\nabla_{\frac{x'}{F}E_1+\frac{y'}{F}E_2}E_3,\frac{y'}{F}E_1-\frac{x'}{F}E_2\rangle=-\frac{l}{2},
\end{eqnarray}
Substituting (\ref{Rict11}) and (\ref{Rict21}) into Equation (16) in
\cite{Ou1}, we have
\begin{equation}\label{Rict31}
\begin{cases}
\kappa''_{g}-\kappa_{g}^3+(4m-l^2)\kappa_{g}=0,\\
3\kappa_{g}\kappa'_{g}=0,\\
-\frac{l}{2}\kappa'_{g}=0.
\end{cases}
\end{equation}
Solving Equation (\ref{Rict31}) we have  $\kappa_{g}=0$ which gives
the minimal surface $\Sigma=\cup_{s\in I}\pi^{-1}(\alpha(s))$, or
$\alpha$ has constant geodesic curvature $\kappa^2_{g}=4m-l^2$. It
follows from \cite{Ou1} (page 229) that the mean curvature of the
Hopf cylinder is given by $H=\frac{\kappa_{g}}{2}$ and
$|A|^2=\kappa^2_{g}+2\tau^2_{g}=4m-\frac{l^2}{2}={\rm constant}$.
From these we conclude that the Hopf cylinder $\Sigma=\cup_{s\in
I}\pi^{-1}(\alpha(s))$ is proper biharmonic if only if
\begin{equation}\label{bz}
\begin{cases}
H^2=\frac{4m-l^2}{4} >0,\\
|A|^2=4m-\frac{l^2}{2}>0.
\end{cases}
\end{equation}

It follows from (\ref{bz}) that  $m >0$ and hence the potential
Bianchi-Cartan-Vranceanu space is either $S^2(1/(2\sqrt{m})\times
\mathbb{R}$ or $SU(2)$ with $m>0$. Applying our characterizations of
Hopf cylinders in $S^2(1/(2\sqrt{m})\times \mathbb{R}$ or $SU(2)$
given in Theorem \ref{CVV} we obtain the Theorem.
\end{proof}
{\bf 2.2 Constant mean curvature biharmonic surfaces in Sol space}\\

Let $(\mathbb{R}^{3},g_{Sol})$ denote Sol space, where the metric
can be written as $g_{Sol}=e^{2z}{\rm d}x^{2}+e^{-2z}{\rm
d}y^{2}+{\rm d}z^{2}$ with respect to the standard coordinates
$(x,y,z)$ in $\mathbb{R}^{3}$. One can easily check that an
orthonormal frame on Sol space can be chosen to be:
\begin{equation}\notag
E_{1}=e^{-z}\frac{\partial}{\partial x},\;
E_{2}=e^{z}\frac{\partial}{\partial y},
\;E_{3}=\frac{\partial}{\partial z}.
\end{equation}
With respect to this orthonormal frame, the Lie brackets and the
Levi-Civita connection can be easily computed as:
\begin{equation}\notag
[E_{1},E_{2}]=0, \;[E_{2},E_{3}]=-E_{2}, \;[E_{1},E_{3}]=E_{1},
\end{equation}
\begin{equation}\notag
\begin{array}{lll}
\nabla_{E_{1}}E_{1}=-E_{3},\hskip0.7cm\nabla_{E_{1}}E_{2}=0, \hskip1cm\nabla_{E_{1}}E_{3}=E_{1}\\
\nabla_{E_{2}}E_{1}=0,\hskip1.2cm \nabla_{E_{2}}E_{2}=E_{3},\hskip1cm\nabla_{E_{2}}E_{3}=-E_{2}\\
\nabla_{E_{3}}E_{1}=0,\hskip1.2cm\nabla_{E_{3}}E_{2}=0,\hskip1.2cm\nabla_{E_{3}}E_{3}=0.\\
\end{array}
\end{equation}
A further computation gives
\begin{equation}\notag
\begin{array}{lll}
R(E_1,E_2)E_1=-E_{2},\; R(E_1,E_3)E_1=E_{3},
R(E_1,E_2)E_2=E_{1},\\R(E_2,E_3)E_2=E_{3},
R(E_1,E_3)E_3=-E_{1},R(E_2,E_3)E_3=-E_{2},
\end{array}
\end{equation}
and the possible nonzero components of the Riemannian curvature:
\begin{equation}\notag
\begin{array}{lll}
 R_{1212}=g(R(E_{1},E_{2})E_{2},E_{1})=1,\\
R_{1313}=g(R(E_{1},E_{3})E_{3},E_{1})=-1,\\
R_{2323}=g(R(E_{2},E_{3})E_{3},E_{2})=-1.
\end{array}
\end{equation}
The Ricci curvature has components:
\begin{equation}\label{00E25}
{\rm Ric}(E_3,E_3)=-2,\;\;{\rm Ric}(E_1,E_1)={\rm Ric}(E_2,E_2)=0.
\end{equation}

\begin{proposition}\label{P1}
A constant mean curvature surface in Sol space is biharmonic if
and only if it is minimal.
\end{proposition}
\begin{proof}
Let $\{e_1=a^iE_i,\;\; e_2=b^iE_i,\;\; \xi=c^iE_i\}$ be an adapted
orthonormal frame with $\xi$ being normal to the surface. Use the
Ricci curvature (\ref{00E25}) we have ${\rm
Ric}\,(\xi,\xi)=-2(c^3)^2, \;\;({\rm
Ric}\,(\xi))^{\top}=-2c^3a^3E_1-2c^3b^3E_2$.
 From these together with biharmonic surface equation and the assumption that the mean curvature $H$ is constant we conclude that  the
surface is biharmonic if and only if
\begin{equation}\notag
\begin{cases}
-H [|A|^{2}+2(c^3)^2]=0,\\
-2c^3a^3H=0,\\
-2c^3b^3H=0,
\end{cases}
\end{equation}
which has solution $H=0$ meaning that the surface is minimal, or
\begin{equation}\label{CSL}
\begin{cases}
|A|^{2}+2(c^3)^2=0,\\
c^3a^3=0,\\
c^3b^3=0.
\end{cases}
\end{equation}
Solving Equations (\ref{CSL}) we have $c^3=0$ and  $|A|^{2}=0$,
which implies the surface is minimal. Thus, we obtain the
proposition.
\end{proof}

\begin{corollary}
The only constant mean curvature proper biharmonic surafces in
Thurston's 3-dimensional geometries are a part of $S^2(1/\sqrt{2})$
in $S^3$, or a part of $S^1(1/\sqrt{2})\times \mathbb{R}$ in
$S^2\times \mathbb{R}$.
\end{corollary}
\begin{proof}
By the classification results of \cite{Ji}, \cite{CI} and
\cite{CMO2}, the only proper biharmonic surface in space forms
$\r^3, H^3$ and $S^3$ is a part of $S^2(1/\sqrt{2})$ in $S^3$. It
follows from Theorem \ref{CVV} and Proposition \ref{P1} that the
only constant mean curvature proper biharmonic surface in $S^2\times
\r,\; H^2\times \r$, $\widetilde{SL}(2,\r)$, ${\rm Nil}$, \; and
$\rm Sol$ spaces is a part of $S^1(1/\sqrt{2})\times \mathbb{R}$ in
$S^2\times \mathbb{R}$. Combining these we obtain the corollary.
\end{proof}

\section{Totally umbilical biharmonic surfaces in 3-dimensional
geometries}

In this section, we first prove that a totally umbilical biharmonic
surface in any $3$-dimensional Riemannian manifold must have
constant mean curvature. We then use this theorem to show that the
only totally umbilical proper biharmonic surface in 3-dimensional
geometries is a part of $S^2(1/\sqrt{2})$ in $S^3$.\\

\begin{theorem}\label{MT44}
A totally umbilical  biharmonic surface  in $3$-dimensional
Riemannian manifolds must have constant mean curvature.
\end{theorem}

\begin{proof}
Take an orthonormal frame $\{e_1=a^iE_i, e_2=b^iE_i,\xi=c^iE_i\}$ of
$3$-dimensional Riemannian manifold adapted to the surface $M$ such
that $A e_i=\lambda_ie_i$, where $A$ is the Weingarten map of the
surface and $\lambda_i$ is the principal curvature in the direction
$e_i$. Since $M$ is supposed to be totally umbilical, i.e., all
principal normal curvatures  at any point of $M$ are equal to the
same number $\lambda$. It follows that
\begin{eqnarray}\label{Ubii}
H=\frac{1}{2}\sum_{i=1}^2\langle A e_i,e_i\rangle=\lambda,\\\notag A
({\rm grad} H)=A(\sum_{i=1}^2(e_i \lambda)e_i)=\frac{1}{2}{\rm
grad}\, \lambda^2,\\\notag |A|^2=2 \lambda^2.
\end{eqnarray}

On the other hand, a straightforward computation  gives
\begin{eqnarray}\label{Cod11}
\langle R(e_1,e_2)e_1,\xi\rangle= R(\xi, e_1,e_1, e_2) =-{\rm
Ric}(e_2,\xi),\\\label{Cod22} \langle R(e_1,e_2)e_2,\xi\rangle
=R(\xi, e_2,e_1, e_2) ={\rm Ric}(e_1,\xi) .
\end{eqnarray}

Noting that $e_1, e_2$ are principal directions with principal
curvature $\lambda$ we can check that
\begin{eqnarray}\label{Cod33}
(\nabla_{e_1}h)(e_2,e_1)&=&{e_1}(h(e_2,e_1))-h(\nabla_{e_1}e_2,e_1)-h(\nabla_{e_1}e_1,e_2)\\\notag
&=&-h(e_1,e_1)\langle\nabla_{e_1}e_2,e_1\rangle-h(e_2,e_2)\langle\nabla_{e_1}e_1,e_2\rangle\\\notag
&=&-\lambda(\langle\nabla_{e_1}e_2,e_1\rangle+\langle\nabla_{e_1}e_1,e_2\rangle)=0,
\end{eqnarray}

\begin{equation}\label{Cod44}
\begin{array}{lll}
(\nabla_{e_2}h)(e_1,e_1)={e_2}(h(e_1,e_1))-h(\nabla_{e_2}e_1,e_1)-h(\nabla_{e_2}e_1,e_1)
=e_2(\lambda),
\end{array}
\end{equation}
\begin{equation}\label{Cod55}
\begin{array}{lll}
(\nabla_{e_1}h)(e_2,e_2)={e_1}(h(e_2,e_2))-h(\nabla_{e_1}e_2,e_2)-h(\nabla_{e_1}e_2,e_2)
=e_1(\lambda),
\end{array}
\end{equation}
and
\begin{equation}\label{Cod66}
\begin{array}{lll}
(\nabla_{e_2}h)(e_1,e_2)={e_2}(h(e_1,e_2))-h(\nabla_{e_2}e_1,e_2)-h(\nabla_{e_2}e_2,e_1)
=0.
\end{array}
\end{equation}
Using (\ref{Cod11}), (\ref{Cod22}), (\ref{Cod33}), (\ref{Cod44}),
(\ref{Cod55}), (\ref{Cod66}) and  the Codazzi equation for a
hypersurface:
\begin{eqnarray}\notag
&&(\nabla_{X}h)(Y,Z)-(\nabla_{Y}h)(X,Z)=(R^N(X,Y)Z)^{\bot} =\langle
R^N(X,Y)Z, \xi\rangle,
\end{eqnarray}
where the covariant derivative of the second fundamental form $h$ is
defined by
$$(\nabla h)(X, Y, Z) = X(h(Y, Z)) - h(\nabla_X Y, Z) - h(Y,\nabla_X
Z),$$
 we have
\begin{equation}\label{Cod77}
\begin{cases}
e_1(\lambda)={\rm Ric}(e_1,\xi),\\
e_2(\lambda)={\rm Ric}(e_2,\xi).
\end{cases}
\end{equation}

On the other hand, using (\ref{Ubii}) and the second equation of
(\ref{BHEq}) we have
\begin{equation}\label{chang}
 2 \lambda\,{\rm grad}\, \lambda
-\, \lambda \,({\rm Ric}\,(\xi, e_1)e_1+{\rm Ric}\,(\xi, e_2)e_2)=0.
\end{equation}
Substituting Equation (\ref{Cod77}) into (\ref{chang}) we have
\begin{equation}\notag
 \lambda\,{\rm grad}\, \lambda
=0,
\end{equation}
from which we conclude that $\lambda$ is a constant. Thus, we obtain
the theorem.
\end{proof}

Note that to classify totally umbilical biharmonic surfaces in
3-dimensional geometries it is enough to know totally umbilical
biharmonic surfaces in Bianchi-Cartan-Vranceanu 3-spaces and in the
hyperbolic 3-space and $\rm Sol$ space. As biharmonic surfaces in
3-dimensional space forms have been classified we need only to
classify totally umbilical biharmonic surfaces in $3$-dimensional
Bianchi-Cartan-Vranceanu spaces and in $\rm Sol$ space. This is done
by the following two corollaries.\\

\begin{corollary}\label{MT4}
A totally umbilical  biharmonic surface  in Sol space is
biharmonic if and only if it is minimal.
\end{corollary}
\begin{proof}
This is a consequence of Theorem \ref{MT44} and Proposition
\ref{P1}.
\end{proof}

\begin{remark}
Note that there are many totally umbilical surfaces in Sol space
(see \cite{ST} for classifications of totally umbilical surfaces in
Sol space and in a more general homogeneous $3$-manifold).
\end{remark}

\begin{corollary}\label{MT5}
A totally umbilical surface in a $3$-dimensional
Bianchi-Cartan-Vranceanu space is proper biharmonic if and only if
it is  part of $S^2(1/\sqrt{2m})$ in $S^3(1/\sqrt{m})$.
\end{corollary}
\begin{proof}
By Theorem \ref{MT44}, a totally umbilical biharmonic surface in a
$3$-dimensional Bianchi-Cartan-Vranceanu space has constant mean
curvature. This, together with Theorem \ref{CVV}, implies that the
only potential totally umbilical proper biharmonic surface in these
spaces are a part of $S^2(1/\sqrt{2m})$ in $S^3(1/\sqrt{m})$, or a
part of a Hopf cylinder. As the latter surface is clearly not
totally umbilical we conclude.
\end{proof}
\begin{rm}
We remark that totally umbilical surfaces in
Bianchi-Cartan-Vranceanu spaces with four-dimensional isometry group
has been classified in \cite{Ve} whilst a classification of such
surfaces in other three-dimensional homogeneous spaces has not yet
appeared in the literatures (see \cite{Ve}).
\end{rm}

Now we can summarize our classification of totally umbilical
biharmonic surfaces in Thurston's 3-dimensional geometries in the
following

\begin{theorem}
A totally umbilical surface in 3-dimensional geometries is proper
biharmonic if and only if it is a part of  $S^2(1/\sqrt{2})$ in
$S^3$.
\end{theorem}
\begin{proof}
Recall that the eight 3-dimensional geometries are: $ \r^3, \; S^3,
\;H^3$,\; $S^2\times \r,\; H^2\times \r$,\; ${\rm Sol}$, \;${\rm
Nil}$,\; and $\widetilde{SL}(2,\r)$. It is well known (see
\cite{Ji}, \cite{CI}, and \cite{CMO2}) that there is no proper
biharmonic surface in $ \r^3, \; H^3$ and that (see \cite{CMO1}) the
only proper biharmonic surface in $S^3$ is a (part of) sphere
$S^2(1/\sqrt{2})$. These, together with Corollaries \ref{MT4} and
\ref{MT5}, give the complete classification.
\end{proof}

\begin{ack}
We would like to thank C. Oniciuc for some invaluable discussions
and email communications related to this work. Especially we are
most grateful to him for his comments and suggestions that help to
improve our classification of the CMC proper biharmonic surface in
$SU(2)$. We also want to thank J. Inoguchi for some useful comments
that help to improve the original manuscript.
\end{ack}


\begin{thebibliography}{99}
\bibitem[Ba1]{Ba1} A. Balmus, {\em Perspectives on biharmonic maps and submanifolds}, Differential geometry,
Proceedings of the VIII International Colloquium, edited by
Jes$\acute{\rm u}$s A. Alvarez L$\acute{\rm o}$pez, Eduardo
Garcia-Rio, 257--265, World Sci. Publ., Hackensack, NJ, 2009.
\bibitem[Ba2]{Ba2} A. Balmus, {\em Biharmonic maps and submanifolds}, Balkan Society of Geometers, Differential Geometry -
Dynamical Systems * Monographs, 2009.
\bibitem[BMO1]{BMO1} A. Balmus, S. Montaldo and C. Oniciuc, {\em
Classification results for biharmonic submanifolds in spheres},
Israel J. Math. 168 (2008), 201--220.
\bibitem[BMO2]{BMO2} A. Balmuus, S. Montaldo, C.
Oniciuc, {\em Biharmonic submanifolds in spcae forms}, Symposium
Valenceiennes (2008), 25--32.
\bibitem[BMO3]{BMO3} A. Balmuus, S. Montaldo, C.
Oniciuc, {\em Classification results and new examples of proper
biharmonic submanifolds in spheres}, Note di Matematica, Note Mat.
1(2008), suppl. n. 1, 49-61.
\bibitem[BDI]{BDI} M. Belkhelfa, F. Dillen, and J. Inoguchi,
{\em Surfaces with parallel second fundamental form in
Bianchi-Cartan-Vranceanu spaces}, PDEs, submanifolds and affine
differential geometry (Warsaw, 2000), 67–87, Banach Center Publ.,
57, Polish Acad. Sci., Warsaw, 2002.
\bibitem[CMO1]{CMO1} R. Caddeo, S.  Montaldo, and C. Oniciuc, {\em Biharmonic submanifolds of
$S\sp 3$}. Internat. J. Math. 12 (2001), no. 8, 867--876.
\bibitem[CMO2]{CMO2} R. Caddeo, S. Montaldo and C. Oniciuc, {\em Biharmonic submanifolds in spheres},
 Israel J. Math. 130 (2002), 109--123.
 \bibitem[CMOP]{CMOP} R. Caddeo, S. Montaldo, C. Oniciuc, and P.Piu, {\em The Euler-Lagrange method for biharmonic curves},
 Mediterr. J. Math. 3 (2006), no. 3-4, 449--465.
\bibitem[CH]{CH} B. Y. Chen, {\em Some open problems and conjectures on submanifolds of finite
type}, Soochow J. Math. 17 (1991), no. 2, 169--188.
\bibitem[CI]{CI} B. Y. Chen and S. Ishikawa, {\em Biharmonic pseudo-Riemannian submanifolds
in pseudo-Euclidean spaces}, Kyushu J. Math. 52 (1998), no. 1,
167--185.
\bibitem[FO]{FO} D. Fetcu  and C. Oniciuc, {\em Explicit Formulas for Biharmonic Submanifolds
in non-Euclidean 3-Spheres}, Abh. Math. Sem. Univ. Hamburg 77
(2007), 179-190.
 \bibitem[IIU]{IIU} T. Ichiyama, J. Inoguchi and H. Urakawa, {\em Classifications and Isolation Phenomena of Bi-Harmonic Maps and Bi-Yang-Mills
Fields},  arXiv:0912.4806, Preprint, 2009.
\bibitem[In]{In}  J. Inoguchi, {\em Submanifolds with harmonic mean curvature vector field in contact 3-manifolds},
Colloq. Math. 100 (2004), no. 2, 163--179.
\bibitem[Ji]{Ji} G. Y. Jiang, {\em Some non-existence theorems of $2$-harmonic isometric immersions into Euclidean spaces
 }, Chin. Ann. Math. Ser. 8A (1987)
376-383.
\bibitem[MO]{MO} S. Montaldo and C. Oniciuc, {\em A short survey on biharmonic maps between Riemannian
manifolds}, Rev. Un. Mat. Argentina 47 (2006), no. 2, 1--22 (2007).
\bibitem[On]{On} C. Oniciuc, {\em Biharmonic maps between Riemannian manifolds} An.
Stiin. Univ. Al. I. Cuza Iasi. Mat. (N.S.) 48 (2002), no. 2,
237--248 (2003).
\bibitem[Ou1]{Ou1} Y. -L.Ou, {\em Biharmonic hypersurfaces in Riemannian
manifolds}, Pacific J. Math., 248 (1), (2010), 217-232.
\bibitem[Ou2]{Ou2} Y. -L. Ou, {\em Some constructions of biharmonic maps and Chen's conjecture on biharmonic hypersurfaces
}, arXiv:0912.1141, preprint 2009.
\bibitem[OT]{OT} Y.-L. Ou and L. Tang, {\em The generalized Chen's
conjecture on biharmonic submanifolds is false}, Preprint,
\bibitem[ST]{ST} R. Souam and E. Toubiana, {\em Totally umbilic surfaces
in homogeneous 3-manifolds}, Comment. Math. Helv. 84 (2009), no. 3,
673--704.
\bibitem[Ta]{Ta} M. Tamura, {\em Gauss maps of surfaces in contact space
forms}, Comment. Math. Univ. St. Pauli 52 (2003), no. 2, 117–123.
\bibitem[Ve]{Ve} J. Van der Veken, {\em Submanifolds of homogeneous
spaces}, Thesis, Katholieke Universiteit Leuven, 2007.
\end{thebibliography}
\end{document}